\documentclass[12pt,reqno]{amsart}
\usepackage{amsmath,amsthm,amsfonts,amssymb,times}
\usepackage{verbatim}
\setbox0=\hbox{$+$}
\newdimen\plusheight
\plusheight=\ht0
\def\+{\;\lower\plusheight\hbox{$+$}\;}

\setbox0=\hbox{$-$}
\newdimen\minusheight
\minusheight=\ht0
\def\-{\;\lower\minusheight\hbox{$-$}\;}

\setbox0=\hbox{$\cdots$}
\newdimen\cdotsheight
\cdotsheight=\plusheight
\def\cds{\lower\cdotsheight\hbox{$\cdots$}}

\renewcommand{\Re}{\operatorname{Re}}

\makeatletter
\def\leqalignno#1{\displ@y \tabskip\z@ plus\@ne fil
  \halign to\displaywidth{\hfil$\@lign\displaystyle{##}$\tabskip\z@skip
    &$\@lign\displaystyle{{}##}$\hfil\tabskip\z@ plus\@ne fil
    &\kern-\displaywidth\rlap{$\@lign\hbox{\rm##}$}\tabskip\displaywidth\crcr
    #1\crcr}}
\makeatother

\newcommand{\eb}{\begin{equation}}
\newcommand{\ee}{\end{equation}}

\newcommand{\df}{\dfrac}
\newcommand{\tf}{\tfrac}

\renewcommand{\Re}{\operatorname{Re}}

\renewcommand{\k}{\kappa}

\renewcommand{\Re}{\text{Re}}

\renewcommand{\(}{\left\(}
\renewcommand{\)}{\right\)}
\renewcommand{\[}{\left\[}
\renewcommand{\]}{\right\]}

\renewcommand{\pmod}[1]{\,(\textup{mod}\,#1)}
\numberwithin{equation}{section}
 \theoremstyle{plain}

\newtheorem{theorem}{Theorem}[section]
\numberwithin{equation}{section}
\theoremstyle{plain}

\newtheorem{conjecture}[theorem]{Conjecture}

\allowdisplaybreaks

\begin{document}

\title[Finite Trigonometric Sums Arising from Ramanujan's Theta Functions]
{Finite Trigonometric Sums Arising from Ramanujan's Theta Functions}
\author{Bruce C.~Berndt, Sun Kim, Alexandru Zaharescu}
\address{Department of Mathematics, University of Illinois, 1409 West Green
Street, Urbana, IL 61801, USA} \email{berndt@illinois.edu}
\address{Department of Mathematics, and Institute of Pure and Applied Mathematics, Jeonbuk National
University, 567 Baekje-daero, Jeonju-si, Jeol labuk-do 54896, Republic of Korea}
\email{sunkim@jbnu.ac.kr}
\address{Department of Mathematics, University of Illinois, 1409 West Green
Street, Urbana, IL 61801, USA; Institute of Mathematics of the Romanian
Academy, P.O.~Box 1-764, Bucharest RO-70700, Romania}
\email{zaharesc@illinois.edu}

\begin{abstract} Two classes of finite trigonometric sums, each involving only $\sin$'s, are evaluated in closed form.  The previous and  original proofs arise from Ramanujan's theta functions and modular equations.
\end{abstract} 

\maketitle

\section{Introduction}

Our paper is motivated by two papers, one by K.~N.~Harshitha, K.~R.~Vasuki, and M.~V.~Yathirajsharma \cite{harshitha} and the other by G.~Vinay, H.~T.~Shwetha, and K.~N.~Harshitha \cite{palestine}, in which they evaluate certain finite trigonometric sums by appealing to Ramanujan's theory of theta functions.  In the former paper, Lambert series, elliptic integrals, and a famous $q$-series theorem of Bailey also play important roles.  The principal idea rests upon the asymptotic behavior of these functions as $q\to 1^-$.  In the latter paper, the authors use modular equations of Ramanujan and the limiting behavior of quotients of theta functions as $q\to 1^-$ to prove their identities.

In \cite{harshitha}, there are four primary theorems giving the values of finite sums of quotients of $\sin$'s \cite[Thm.~1.1, (1.7)--(1.10)]{harshitha}.  We shall provide short proofs of these identities with the use of contour integration.  In doing so, we shall also simplify their theorems.  At the end of their paper, they offer two conjectures.  We shall prove their conjectures by establishing a general theorem, which yields their conjectures as special cases.

In \cite{palestine}, the authors employ modular equations to evaluate seven trigonometric sums of a different sort.  We show that these can be derived in an elementary manner. However, it should be emphasized that although elementary proofs of the seven identities can now be provided, their existence depends upon the authors of \cite{palestine} discovering them via Ramanujan's modular equations. Surprisingly, one of the identities was more intractable than the others, and Gauss's cyclotomy was invoked to provide our proof.



\section{Four Identities from \cite{harshitha}}

\begin{theorem}\label{thmh1} Let $n$ be  an even natural number, and let $ j \equiv 2\pmod4$, where $(\tf12 j, \tf12 n)=1$.  Then
\begin{equation}\label{hh1}
\sum_{k=1}^{\tf12 n-1} \df{\sin\left(\dfrac{ (j-1)k\pi}{n}\right)\sin\left(\dfrac{ (j+1)k\pi}{n}\right)}
{\sin^2\left(\dfrac{ k\pi}{n}\right)\sin^2\left(\dfrac{jk\pi}{n}\right)}=\df{n^2-4}{12}.
\end{equation}
\end{theorem}

\begin{theorem}\label{thmh2} Let $n$ denote an even positive integer, and let $ j \equiv 2\pmod4$, where $(\tf12 j, \tf12 n)=1$.  Then
\begin{equation}\label{hh9}
\sum_{k=0}^{\tf12 n-1} \df{\sin\left(\dfrac{ (j-1)(2k+1)\pi}{2n}\right)\sin\left(\dfrac{ (j+1)(2k+1)\pi}{2n}\right)}
{\sin^2\left(\dfrac{ (2k+1)\pi}{2n}\right)\sin^2\left(\dfrac{j(2k+1)\pi}{2n}\right)}=\df{n^2}{4}.
\end{equation}
\end{theorem}

(In \cite[Thm.~1.1, no.~3]{harshitha}, replace $p$ by $j$.)

\begin{theorem}\label{thmh3} Let $n$ be  an odd natural number, and let $ j=2p$ be an even positive integer, where $(j,n)=1$.  Then
\begin{equation}\label{hh11}
\sum_{k=0}^{\tf{n-3}{2}} \df{\sin\left(\dfrac{ (j-1)(2k+1)\pi}{2n}\right)\sin\left(\dfrac{ (j+1)(2k+1)\pi}{2n}\right)}
{\sin^2\left(\dfrac{ (2k+1)\pi}{2n}\right)\sin^2\left(\dfrac{j(2k+1)\pi}{2n}\right)}=\df{n^2-1}{3}.
\end{equation}
\end{theorem}

\begin{theorem}\label{thmh4}
Let $n$ be an odd positive integer, and let $j$ be a positive integer such that $(j,n)=1$ and $j(2k + 1)$ is not
a multiple of $n$, where $0\leq k \leq n-1$.  Then
\begin{equation}\label{hh14}
\sum_{k=0}^{\tf{n-3}{2}} \df{\sin\left(\dfrac{ (j-1)(2k+1)\pi}{n}\right)\sin\left(\dfrac{ (j+1)(2k+1)\pi}{n}\right)}
{\sin^2\left(\dfrac{ (2k+1)\pi}{n}\right)\sin^2\left(\dfrac{j(2k+1)\pi}{n}\right)}=0.
\end{equation}
\end{theorem}

In each of the theorems above, because of the given hypotheses on $n$ and $j$, the value of the sum is independent of the residue class modulo $n$ over which the index of summation runs.  Note that this is evinced by the fact that the values of the four sums in Theorems \ref{thmh1}--\ref{thmh4} are independent of $j$.  Thus, in our proofs, we fix the value of $j$ without loss of generality.  Moreover, by replacing the summation index $k$ by $n-k$ or $k-(n-1)$ (depending on the sum), the value of the sum over a full residue system modulo $n$ is twice that over a half-residue system.

In each proof, we integrate over a positively oriented rectangle $C_N$ with vertical sides passing through $0$ and $n$ and horizontal sides $z=x+\pm iN, 0\leq x\leq n$.  Furthermore, $C_N$ has semicircular indentations of radius $\epsilon, 0<\epsilon<\tf12$, about $z=0$ and $z=n$, with the former indentation in the left half-plane and the latter indentation lying to the left of the line $z=n+iy, -N\leq y \leq N$.  In each case, we evaluate the integral of a function $f(z)$ over $C_N$ with the use of the residue theorem.  Throughout,  $R_a$ denotes the residue of the given function $f(z)$ at a pole $z=a$.

 We next evaluate the integral directly and then let $N\to\infty$.  Each of the functions $f(z)$ has period $n$, and consequently the integrals over the vertical sides of $C_N$ cancel.
 Second, using the definition
$$\sin \pi z=\frac{1}{2i}(e^{\pi i z}-e^{-\pi iz}),$$
we can easily show that the integrals over the horizontal sides of $C_N$ tend to $0$ as $N\to\infty$.  In summary,
\begin{equation}\label{hh6}
\lim_{N\to\infty}\int_{C_N}f(z)dz=0.
\end{equation}
By combining the two evaluations, we deduce the identities in our four theorems above.

Since these details are the same for each of our proofs, we mention them only briefly, if at all.

\begin{proof}[Proof of Theorem 2.1.] 

 Because $ j \equiv 2\pmod4$ and $(\tf12 j, \tf12 n)=1$, the sum on the left-hand side of \eqref{hh1} is independent of the half-residue system $ jk$, $0\leq k\leq \tf12 n-1$.  Therefore, it suffices to prove Theorem \ref{thmh1} for $j=2$. As indicated above, the value of the sum in \eqref{hh1} is one-half of the value of the sum for which $1\leq k\leq n-1, k\neq \tf12 n$.

 Let
\begin{equation}\label{hh2}
f(z):= \df{\sin\left(\dfrac{ \pi z}{n}\right)\sin\left(\dfrac{3\pi z}{n}\right)}
{\sin^2\left(\dfrac{\pi z}{n}\right)\sin^2\left(\dfrac{2\pi z}{n}\right)(e^{2\pi iz}-1)}
=\df{\sin\left(\dfrac{3\pi z}{n}\right)}
{\sin\left(\dfrac{\pi z}{n}\right)\sin^2\left(\dfrac{2\pi z}{n}\right)(e^{2\pi iz}-1)}.
\end{equation}
The function $f(z)$ is analytic on the interior of $C_N$, except for a triple pole at $z=0$, simple poles at $z=1,2, \dots, n-1$, and a double pole at $z=\tf12 n$. Let $R_a$ denote the residue of $f(z)$ at a pole $z=a$.  By the residue theorem,
\begin{align}\label{hh3}
\int_{C_N}f(z)dz=&\sum_{\substack{k=1\\k\neq n/2}}^{n-1} \df{\sin\left(\dfrac{k\pi}{n}\right)\sin\left(\dfrac{ 3k\pi}{n}\right)}
{\sin^2\left(\dfrac{ k\pi}{n}\right)\sin^2\left(\dfrac{2k\pi}{n}\right)}+2\pi i(R_0+R_{n/2})\notag\\
=&2\sum_{k=1}^{\tf12 n-1} \df{\sin\left(\dfrac{k\pi}{n}\right)\sin\left(\dfrac{ 3k\pi}{n}\right)}
{\sin^2\left(\dfrac{ k\pi}{n}\right)\sin^2\left(\dfrac{2k\pi}{n}\right)}+2\pi i(R_0+R_{n/2}).
\end{align}
Calculations, by hand or via \emph{Mathematica}, easily give
\begin{equation}\label{hh4}
2\pi i\, R_0=-\df{n^2}{4} \quad\text{and}\quad 2\pi i\, R_{n/2}=\df{n^2+8}{12}.
\end{equation}
Hence, from \eqref{hh3} and \eqref{hh4},
\begin{equation}\label{hh5}
\int_{C_N}f(z)dz=2\sum_{k=1}^{\tf12 n-1} \df{\sin\left(\dfrac{k\pi}{n}\right)\sin\left(\dfrac{ 3k\pi}{n}\right)}
{\sin^2\left(\dfrac{ k\pi}{n}\right)\sin^2\left(\dfrac{2k\pi}{n}\right)}+\df{-n^2+4}{6}.
\end{equation}

Combining \eqref{hh5} and \eqref{hh6}, we readily deduce \eqref{hh1} with $j=2$.
\end{proof}

\begin{proof}[Proof of Theorem 2.2.] 

 On the interior of $C_N$,
\begin{equation}\label{hh7}
f(z):= \df{\sin\left(\dfrac{ \pi z}{n}\right)\sin\left(\dfrac{3\pi z}{n}\right)}
{\sin^2\left(\dfrac{\pi z}{n}\right)\sin^2\left(\dfrac{2\pi z}{n}\right)(e^{2\pi iz}+1)}\notag
=\df{\sin\left(\dfrac{3\pi z}{n}\right)}
{\sin\left(\dfrac{\pi z}{n}\right)\sin^2\left(\dfrac{2\pi z}{n}\right)(e^{2\pi iz}+1)}
\end{equation}
has a double pole at $z=0$, simple poles at $z=\tf12 (2k+1), 0\leq k \leq \tf12 n-1$, and a double pole at $z=\tf12 n$. Thus, by the residue theorem,
\begin{equation}\label{hh8}
\int_{C_N}f(z)dz=
-2\sum_{k=0}^{\tf12 n-1} \df{\sin\left(\dfrac{ (2k+1)\pi}{2n}\right)\sin\left(\dfrac{ 3(2k+1)\pi}{2n}\right)}
{\sin^2\left(\dfrac{ (2k+1)\pi}{2n}\right)\sin^2\left(\dfrac{2(2k+1)\pi}{2n}\right)}+2\pi i(R_0+R_{n/2}).
\end{equation}
Using hand computation or \emph{Mathematica}, we find that
\begin{equation}\label{hh10}
2\pi i\, R_0=\df{3n^2}{4}\quad \text{and}\quad 2\pi i\,R_{n/2} =-\df{n^2}{4}.
\end{equation}
Putting \eqref{hh10} into \eqref{hh8}, using \eqref{hh6},
and  simplifying, we deduce \eqref{hh9} with $j=2$.
\end{proof}

\begin{proof}[Proof of Theorem 2.3.]

 Let $j=2$ and
\begin{equation}\label{hh12}
f(z):= \df{\sin\left(\dfrac{ \pi z}{n}\right)\sin\left(\dfrac{3\pi z}{n}\right)}
{\sin^2\left(\dfrac{\pi z}{n}\right)\sin^2\left(\dfrac{2\pi z}{n}\right)(e^{2\pi iz}+1)}.
\end{equation}
The function $f(z)$ has a pole of order 2 at the origin, simple poles at $z=\tf12 (2k+1), 0\leq k \leq n-1$, and a pole of order 3 at $z=\tf12 n$.
Applying the residue theorem, we deduce that
\begin{equation}\label{hh13}
\int_{C_N}f(z)dz=
-2\sum_{k=0}^{\tf{n-3}{2}} \df{\sin\left(\dfrac{ (2k+1)\pi}{2n}\right)\sin\left(\dfrac{ 3(2k+1)\pi}{2n}\right)}
{\sin^2\left(\dfrac{ (2k+1)\pi}{2n}\right)\sin^2\left(\dfrac{(2k+1)\pi}{n}\right)}+\df{3n^2}{4}-\df{n^2+8}{12}.
\end{equation}
Employing \eqref{hh6}
and \eqref{hh13}, and simplifying, we complete the proof of \eqref{hh11}.
\end{proof}

\begin{proof}[Proof of Theorem 2.4.] 

Since $(j,n)= 1$ and $j(2k+1)$,  is not a multiple of $n$,  $0\leq k \leq n-1$,  $j(2k+1)$ runs over the complete residue system  $0\leq k \leq n-1$ modulo $n$ for any such $j$.   In other words, the value of the sum in Theorem \ref{thmh4} taken over $0\leq k\leq n-1$ is independent of the value of $j$.  To reiterate what we emphasized above, the terms with index $k$ and $n-k-1$ are identical. Thus, if we take $j=1$, the sum in \eqref{hh14} is trivially equal to 0.
\end{proof}

\section{Two Conjectures from \cite{harshitha}}
In \cite{harshitha}, the authors made the following two conjectures:
\begin{conjecture}\label{conjecture1}
\begin{align}
\lim_{k\to\infty} S_1:=\sum_{j=0}^{2k-1}(-1)^j\sin^2\left(\df{(2j+1)\pi}{8k+2}\right)=-\df12,\label{h1}\\
\lim_{k\to\infty} S_2:=\sum_{j=0}^{2k}(-1)^j\sin^2\left(\df{(2j+1)\pi}{8k+6}\right)=\df12.\label{h2}
\end{align}
\end{conjecture}

 The truth of Conjecture \ref{conjecture1} is a consequence of the following theorem.

\begin{theorem}\label{abcde} Let $a,b$, and $c$ be positive.  Then
\begin{equation}\label{abc}
S(a,b,c):=\sum_{j=0}^{2k-1}(-1)^j\sin^2\left(\df{(2j+1)\pi a}{bk+c}\right)=
-\df{\sin^2\left(\df{4\pi ak}{bk+c}\right)}{2\cos\left(\df{2\pi a}{bk+c}\right)}.
\end{equation}
\end{theorem}

If we let $a=1, b=8$, and  $c=2$, the sum \eqref{abc} reduces to $S_1$ in \eqref{h1}.  Letting $k\to\infty$, we immediately deduce \eqref{h1}.  Next, let $a=1$, $b=8$, and $c=6$.  If we add the summand with $j=2k$ to the sum in \eqref{abc} with $a=1$, $b=8$, and $c=6$, we obtain the sum in \eqref{h2}.  Since this `extra' term tends to 1 as $k\to\infty$, \eqref{h2} readily follows from Theorem \ref{abcde}.

\begin{proof}    First,
\begin{align}\label{t1}
\sin^2(u+t)-\sin^2(u-t)=&(\sin u \cos t+\sin t \cos u)^2 -(\sin u \cos t-\sin t \cos u)^2\notag\\
=&4\sin u \cos u \,\,\sin t\cos t =\sin 2u\,\,\sin 2t.
\end{align}
 With
\begin{equation*}
u=\df{(2j+2)\pi a}{bk+c}\quad \text{and}\quad t=\df{\pi a}{bk+c}
\end{equation*}
in \eqref{t1}, we find that
\begin{align}\label{abcd}
S(a,b,c)=&-\sum_{\substack{j=0\\j\,\,\text{even}}}^{2k-1}
\left(\sin^2\left(\df{(2j+1)\pi a}{bk+c}\right)-\sin^2\left(\df{(2j+3)\pi a}{bk+c}\right)\right)\notag\\
=&-\sin\left(\df{2\pi a}{bk+c}\right)\sum_{\substack{j=0\\j\,\,\text{even}}}^{2k-1}\sin\left(\df{(4j+4)\pi a}{bk+c}\right)\notag\\
=&-\sin\left(\df{2\pi a}{bk+c}\right)\sum_{j=0}^{k-1}\sin\left(\df{(8j+4)\pi a}{bk+c}\right).
\end{align}
Recall \cite[p.~36, formula 1.341, no.~1]{gr}
\begin{equation}\label{t3}
\sum_{n=0}^{k-1}\sin(x+ny)=\sin\left(x+\df{(k-1)y}{2}\right)\sin\left(\df{ky}{2}\right)\csc\left(\df{y}{2}\right).
\end{equation}
Applying \eqref{t3} with
\begin{equation*}
x=\df{4\pi a}{bk+c}\quad\text{and}\quad y=\df{8\pi a}{bk+c},
\end{equation*}
we deduce from \eqref{abcd} that
\begin{align*}
S(a,b,c)=&-\sin\left(\df{2\pi a}{bk+c}\right)\sin^2\left(\df{4\pi ak}{bk+c}\right)\csc\left(\df{4\pi a}{bk+c}\right)\\
=&-\df{\sin^2\left(\df{4\pi ak}{bk+c}\right)}{2\cos\left(\df{2\pi a}{bk+c}\right)},
\end{align*}
which completes the proof of \eqref{abc}.
\end{proof}

\section{Elementary Trigonometric Sums Arising from Ramanujan's Modular Equations}

In \cite{palestine}, G.~Vinay, H.~T.~Shwetha, and K.~N.~Harshitha evaluated 7 sums of quotients of $\sin$'s by taking 7 modular equations of Ramanujan, recasting them in terms of theta functions, and then calculating their limits as $q\to 1^-$.  One of the identities arises from a modular equation of degree 15, two of them depend on modular equations of degree 17, and four rest upon modular equations of degree 13.  We show that each of them can be proved in an elementary fashion, although one of those of degree 13 is more recondite than the others, and needed Gauss's theory of cyclotomy to prove it.

\begin{theorem}\label{thm1.7} We have
\begin{equation}\label{1.7}
L_1(15):=\df{\sin(2\pi/15)}{\sin(\pi/15)}-\df{\sin(7\pi/15)}{\sin(4\pi/15)}-\df{\sin(3\pi/15)}{\sin(6\pi/15)}=0.
\end{equation}
\end{theorem}

\begin{proof}   Rewrite $L_1(15)$ and then use the `double angle' formula for $\sin$ and the `addition' formula for $\cos$. Thus, we find that
\begin{align*}
L_1(15)=&\df{\sin(2\pi/15)}{\sin(\pi/15)}-\df{\sin(8\pi/15)}{\sin(4\pi/15)}-\df{\sin(12\pi/15)}{\sin(6\pi/15)}\\
=&2\cos(\pi/15)-2\cos(4\pi/15)-2\cos(6\pi/15)\\
=&2\cos(\pi/15)-4\cos(5\pi/15)\cos(\pi/15)\\
=&2\cos(\pi/15)\{1-2\cos(\pi/3)\}=2\cos(\pi/15)(1-2\cdot\tf12)=0.
\end{align*}
Hence, \eqref{1.7} has been proved.
\end{proof}

\begin{theorem}\label{thm1.8}  We have
\begin{align*}
L_1(17):=&\df{\sin(6\pi/17)}{\sin(3\pi/17)}-\df{\sin(4\pi/17)}{\sin(2\pi/17)}-\df{\sin(8\pi/17)}{\sin(4\pi/17)}+\df{\sin(2\pi/17)}{\sin(\pi/17)}
\notag\\
&+\df{\sin(7\pi/17)}{\sin(5\pi/17)}-\df{\sin(5\pi/17)}{\sin(6\pi/17)}+\df{\sin(3\pi/17)}{\sin(7\pi/17)}-\df{\sin(\pi/17)}{\sin(8\pi/17)}=1.
\end{align*}
\end{theorem}

\begin{proof}  Using the `double angle' formula for $\sin$, we deduce that
\begin{align*}
L_1(17)=&2\left\{\cos(3\pi/17)-\cos(2\pi/17)-\cos(4\pi/17)+\cos(\pi/17)\right.\notag\\
&+\left.\cos(5\pi/17)-\cos(6\pi/17)+\cos(7\pi/17)-\cos(8\pi/17)\right\}\notag\\
=&-2\sum_{n=1}^8(-1)^n\cos(n\pi/17)\\
=&-\Re\left\{\sum_{n=1}^{16}(-1)^ne^{\pi in/17}\right\}\\
=&\df{e^{\pi i/17}(1-e^{16\pi i/17})}{1+e^{\pi i/17}}=\df{e^{\pi i/17}+1}{1+e^{\pi i/17}}=1,
\end{align*}
which completes the proof of Theorem \ref{thm1.8}.
\end{proof}

\begin{theorem}\label{thm1.9} We have
\begin{align}\label{1.9}
L_2(17):=&\df{\sin(6\pi/17)\sin(7\pi/17)}{\sin(3\pi/17)\sin(5\pi/17)}+\df{\sin(4\pi/17)\sin(\pi/17)}{\sin(2\pi/17)\sin(8\pi/17)}\notag\\
-&\df{\sin(8\pi/17)\sin(2\pi/17)}{\sin(4\pi/17)\sin(\pi/17)}-\df{\sin(5\pi/17)\sin(3\pi/17)}{\sin(6\pi/17)\sin(7\pi/17)}=-1.
\end{align}
\end{theorem}

\begin{proof} Utilizing first the `double angle' formula for $\sin$, and next the `addition' formula for $\cos$, we find that
\begin{align*}
L_2(17)=&4\cos(3\pi/17)\cos(5\pi/17)+4\cos(2\pi/17)\cos(8\pi/17)\\
&-4\cos(4\pi/17)\cos(\pi/17)-4\cos(6\pi/17)\cos(7\pi/17)\\
=&2\left\{\cos(8\pi/17)+\cos(2\pi/17)+\cos(10\pi/17)+\cos(6\pi/17)\right.\\
&-\left.\cos(5\pi/17)-\cos(3\pi/17)-\cos(13\pi/17)-\cos(\pi/17)\right\}\\
=&2\sum_{n=1}^8(-1)^n\cos(n\pi/17)=-1,
\end{align*}
by the same argument that was used in the proof of Theorem \ref{thm1.8}.
This completes the proof of Theorem \ref{thm1.9}.
\end{proof}

\begin{theorem}\label{thm1.3}
We have
\begin{align}\label{1.3}
L_1(13):=&\df{\sin(4\pi/13)}{\sin(2\pi/13)}-\df{\sin(6\pi/13)}{\sin(3\pi/13)}-\df{\sin(2\pi/13)}{\sin(\pi/13)}\notag\\
&+\df{\sin(5\pi/13)}{\sin(4\pi/13)}-\df{\sin(3\pi/13)}{\sin(5\pi/13)}+\df{\sin(\pi/13)}{\sin(6\pi/13)}=-1.
\end{align}
\end{theorem}

\begin{proof} By the `double angle' formula, we write the sum in \eqref{1.3} as
\begin{align}\label{a}
L_1(13)=&2\left\{\cos(2\pi/13)-\cos(3\pi/13)-\cos(\pi/13)\right.\notag\\
&\left.+\cos(4\pi/13)-\cos(5\pi/13)+\cos(6\pi/13)\right\}\notag\\
=&2\sum_{n=1}^6(-1)^n\cos(n\pi/13)\notag\\
=&\Re\left\{\sum_{n=1}^{12}(-1)^ne^{\pi in/13}\right\}\notag\\
=&\df{-e^{\pi i/13}(1-e^{12\pi i/13})}{1+e^{\pi i/13}}=-1,
\end{align}
which finishes the proof of \eqref{1.3}.
\end{proof}

\begin{theorem}\label{thm1.4}  We have
\begin{equation}\label{1.4}
L_2(13):=\df{\sin(4\pi/13)\sin(6\pi/13)}{\sin(2\pi/13)\sin(3\pi/13)}-\df{\sin(2\pi/13)\sin(3\pi/13)}{\sin(\pi/13)\sin(5\pi/13)}
-\df{\sin(5\pi/13)\sin(\pi/13)}{\sin(4\pi/13)\sin(6\pi/13)}=1.
\end{equation}
\end{theorem}

\begin{proof}   Using the `double angle' formula followed by the addition formula for $\cos$, we find that
\begin{align*}
&L_2(13)=4\{\cos(2\pi/13)\cos(3\pi/13)-\cos(\pi/13)\cos(5\pi/13)-\cos(4\pi/13)\cos(6\pi/13)\}\\
=&2\{\cos(5\pi/13)+\cos(\pi/13)-\cos(4\pi/13)-\cos(6\pi/13)-\cos(10\pi/13)-\cos(2\pi/13)\}\\
=&2\{\cos(5\pi/13)+\cos(\pi/13)-\cos(4\pi/13)-\cos(6\pi/13)+\cos(3\pi/13)-\cos(2\pi/13)\}\\
=&-2\sum_{n=1}^6(-1)^n\cos(n\pi/13)
=1,
\end{align*}
by \eqref{a}.
\end{proof}

\begin{theorem}\label{thm1.5} We have
\begin{align}\label{1.5}
L_3(13):=\df{\sin(2\pi/13)\sin(3\pi/13)}{\sin(4\pi/13)\sin(6\pi/13)}-\df{\sin(\pi/13)\sin(5\pi/13)}{\sin(2\pi/13)\sin(3\pi/13)}-
\df{\sin(4\pi/13)\sin(6\pi/13)}{\sin(5\pi/13)\sin(\pi/13)}=-4.
\end{align}
\end{theorem}

\begin{proof}
By  the `double angle' formula for $\sin$ and a formula for the product of the six $\cos$'s \cite[p.~41, Formula 1.392, no.~1]{gr}, we find that
\begin{align}\label{1.5a}
L_3(13)=&\df{1}{4\cos(2\pi/13)\cos(3\pi/13)}-\df{1}{4\cos(\pi/13)\cos(5\pi/13)}-\df{1}{4\cos(4\pi/13)\cos(6\pi/13)}\notag\\
=&\df14\left\{{\prod_{n=1}^6\cos(n\pi/13)}\right\}^{-1}\bigg(\cos(\pi/13)\cos(5\pi/13)\cos(4\pi/13)\cos(6\pi/13)\notag\\
&-\cos(2\pi/13)\cos(3\pi/13)\cos(4\pi/13)\cos(6\pi/13)\notag\\
&-\cos(2\pi/13)\cos(3\pi/13)\cos(\pi/13)\cos(5\pi/13)\bigg)\notag\\
=&2^4\bigg(\cos(\pi/13)\cos(5\pi/13)\cos(4\pi/13)\cos(6\pi/13)\notag\\&-\cos(2\pi/13)\cos(3\pi/13)\cos(4\pi/13)\cos(6\pi/13)\notag\\
&-\cos(2\pi/13)\cos(3\pi/13)\cos(\pi/13)\cos(5\pi/13)\bigg).
\end{align}
Repeated use of the elementary formula for the product of two $\cos$'s gives
\begin{align*}
&\cos a\cos b \cos c \cos d\\ =&\df18\left\{\cos(a+b+c+d)+\cos(a+b+c-d)+\cos(a+b-c+d)+\cos(a+b-c-d)\right.\\
&\left.+\cos(a-b+c+d)+\cos(a-b+c-d)+\cos(a-b-c+d)+\cos(a-b-c-d)\right\}.
\end{align*}
Hence, with $x=\pi/13$,
\begin{align*}
\cos x\cos 5x \cos 4x\cos 6x =& \df18\left\{-\cos x+ \cos 2x -\cos 3x +2\cos 4x-\cos 5x +2\cos6x\right\},\\
\cos 2x\cos 3x \cos 4x\cos 6x =& \df18\left\{\cos x-2 \cos 2x +2\cos 3x -\cos 4x+\cos 5x -\cos6x\right\},\\
\cos 2x\cos 3x \cos x\cos 5x =& \df18\left\{2\cos x- \cos 2x +\cos 3x -\cos 4x+2\cos 5x -\cos6x\right\}.\\
 \end{align*}
 Using these identities in \eqref{1.5a},  we conclude that
 \begin{align*}
  L_2(13)=&2^3\left\{-\cos(\pi/13) +\cos (2\pi/13)-\cos (3\pi/13)\right.\\&\left.+\cos (4\pi/13)-\cos (5\pi/13)+\cos (6\pi/13)\right\}\\
  =&2^3 \sum_{n=1}^6(-1)^n\cos(n\pi/13)
  =-4,
  \end{align*}
  by \eqref{a}, and so the proof is complete.
\end{proof}

\begin{theorem}\label{thm1.6} We have
\begin{equation}\label{1.6}
\df{\sin(6\pi/13)\sin(2\pi/13)\sin(5\pi/13)}{\sin(4\pi/13)\sin(3\pi/13)\sin(\pi/13)}-
\df{\sin(4\pi/13)\sin(3\pi/13)\sin(\pi/13)}{\sin(6\pi/13)\sin(2\pi/13)\sin(5\pi/13)}=3.
\end{equation}
\end{theorem}

\begin{proof}  Let
\begin{equation}\label{1.6a}
P:=\df{\sin(6\pi/13)\sin(2\pi/13)\sin(5\pi/13)}{\sin(4\pi/13)\sin(3\pi/13)\sin(\pi/13)}.
\end{equation}
Thus, by \eqref{1.6}, we want to prove that
\begin{equation}\label{1.6b}
P-\df{1}{P}=3.
\end{equation}

Examining the denominator  of $P$, we note that $1,3,4$ are quadratic residues modulo $13$, and examining the numerator of $P$, we note that $2,5,6$ are quadratic non-residues modulo $13$. Since $13\equiv 1\pmod4$, $12, 10$, and $9 $ are quadratic residues modulo $13$.  Similarly, $7,8$, and $11$ are quadratic non-residues modulo $13$.  In conclusion, we see that
\begin{equation}\label{1.6d}
Q:=P^2=\df{\prod_{n}\sin(n\pi/13)}{\prod_{r}\sin(r\pi/13)},
\end{equation}
where in the product in the numerator, $n$ runs through a complete set of quadratic non-residues modulo 13, and where in the product in the denominator $r$ runs through a complete set of quadratic residues modulo 13.

From Davenport's book \cite[p.~10]{davenport}, by a result of Gauss in cyclotomy, there exist polynomials $Y(x)$ and $Z(x)$ with integral coefficients such that
\begin{align*}
\prod_r(x-e^{2\pi ir/13})=&\frac12\left\{Y(x)-\sqrt{13}Z(x)\right\},\\
\prod_n(x-e^{2\pi in/13})=&\frac12\left\{Y(x)+\sqrt{13}Z(x)\right\},
\end{align*}
where $r$ and $n$ have the same meanings as in \eqref{1.6d}. Furthermore, if $Y:=Y(1)$ and $Z=Z(1)$, then \cite[p.~11]{davenport}
\begin{equation}\label{1.6e}
Q=\df{Y+\sqrt{13}\,Z}{Y-\sqrt{13}\,Z}.
\end{equation}
There remains the calculation of $Y$ and $Z$.

We are grateful to Di Liu, who used \emph{Mathematica} to determine $Y(x)$ and $Z(x)$ for us, to wit,
\begin{align*}
Y(x)&=2x^6+x^5+4x^4-x^3+4x^2+x+2,\\
Z(x)&=x^5+x^3+x.
\end{align*}
So, $Y(1)=13$ and $Z(1)=3$. Thus, by \eqref{1.6e},
\begin{equation}\label{1.6g}
Q=\df{13+3\sqrt{13}}{13-3\sqrt{13}}.
\end{equation}
By \eqref{1.6d} and \eqref{1.6g}, write
\begin{equation*}
P=\sqrt{\df{13+3\sqrt{13}}{13-3\sqrt{13}}}=A+B\sqrt{13}.
\end{equation*}
An elementary calculation shows that
\begin{equation*}
P=\df{3+\sqrt{13}}{2},
\end{equation*}
which is the fundamental unit for $\mathbb{Q}(\sqrt{13}).$   Hence, \eqref{1.6b} follows, and consequently the proof of \eqref{1.6} is complete.

\end{proof}

\end{document}